\documentclass{amsart}
\usepackage[margin=1in]{geometry}
\usepackage{xcolor}
\usepackage{blkarray}
\usepackage{graphicx}
\usepackage{latexsym,wasysym, mathrsfs, mathtools,hhline,amssymb,amsthm}
\usepackage{algorithm}
\usepackage{bm}
\usepackage{algpseudocode}
\usepackage{enumerate}
\usepackage[capitalise,noabbrev]{cleveref}
\let\oldtextit\textit 
\renewcommand\emph[1]{\oldtextit{\color{RoyalBlue}#1}}
\definecolor{RoyalBlue}{cmyk}{1, 0.50, 0, 0}
\theoremstyle{definition}
\newtheorem{thm}{Theorem}[section]
\newtheorem{example}{Example}[section]
\newtheorem{definition}[example]{Definition}
\newtheorem{remark}[example]{Remark}

\newtheorem{lemma}[example]{Lemma}
\newtheorem{corollary}[example]{Corollary}

\newcommand{\CC}{{\mathbb C}}

\newcommand{\NN}{{\mathbb N}}

\newcommand{\cF}{{\mathcal F}}
\newcommand{\cG}{{\mathcal G}}
\newcommand{\cH}{{\mathcal H}}
\newcommand{\cP}{{\mathcal P}}
\newcommand{\cQ}{{\mathcal Q}}
\newcommand{\cA}{{\mathcal A}}
\newcommand{\cB}{{\mathcal B}}
\newcommand{\cC}{{\mathcal C}}

\newcommand{\bS}{{\mathbb S}}

\newcommand{\smallEpsilon}{{\varepsilon_-}}
\newcommand{\largeEpsilon}{{\varepsilon_+}}
\newcommand{\smallR}{{R_-}}
\newcommand{\largeR}{{R_+}}

\newcommand{\mainfilecheck}[1]{1}

\ifthenelse{\mainfilecheck{1} > 0}{\title[Isolating clusters of zeros using arbitrary-degree inflation]{Isolating clusters of zeros of analytic systems using arbitrary-degree inflation}}{\title{Isolating clusters of zeros of analytic systems using arbitrary-degree inflation}}

\ifthenelse{\mainfilecheck{1} > 0}{\author{Michael Burr}
	\address{School of Mathematical and Statistical Science, Clemson University, 220 Parkway Drive, Clemson, SC 29634}
	\email {burr2@clemson.edu}\urladdr{https://cecas.clemson.edu/~burr2/}

	\author{Kisun Lee}
	\address{Department of Mathematics, University of California San Diego, 9500 Gilman Drive, La Jolla, CA 92093}
	\email {kil004@ucsd.edu}\urladdr{https://klee669.github.io}
	
	\author{Anton Leykin}
	\address{School of Mathematics, Georgia Institute of Technology, 686 Cherry Street, Atlanta, GA 30308}
	\email {leykin@math.gatech.edu}\urladdr{https://antonleykin.math.gatech.edu}}{
	\author{Michael Burr}
	\affiliation{%
		\institution{Clemson University}
		\streetaddress{O-110 Martin Hall, Box 340975}
		\city{Clemson} 
		\state{South Carolina} 
		\country{United States}
		\postcode{29634}
	}
	\email{burr2@clemson.edu}
	
	\author{Kisun Lee}
	\affiliation{%
		\institution{University of California San Diego}
		\streetaddress{9500 Gilman Drive}
		\city{La Jolla} 
		\state{California} 
		\country{United States}
		\postcode{92093}
	}
	\email{kil004@ucsd.edu}
	
	\author{Anton Leykin}
	\affiliation{%
		\institution{Georgia Institute of Technology}
		\streetaddress{686 Cherry Street}
		\city{Atlanta} 
		\state{Georgia} 
		\country{United States}
		\postcode{30332-0160}
	}
	\email{leykin@math.gatech.edu}}

\date{}

\begin{document}
	
	\begin{abstract}
		Given a system of analytic functions and an approximation to a cluster of zeros, we wish to construct two regions containing the cluster and no other zeros of the system.  The smaller region tightly contains the cluster while the larger region separates it from the other zeros of the system.  We achieve this using the method of inflation which, counterintuitively, relates it to another system that is more amenable to our task but whose associated cluster of zeros is larger.  
	\end{abstract}
	
	\ifthenelse{\mainfilecheck{1} > 0}{
	}
	{\settopmatter{printfolios=true}}
	\maketitle
	
	\section{Introduction}
	Suppose that $\cF$ is a system of $m$ \emph{analytic function} in $n$ unknowns, where $m\geq n$, and $z^\ast\in\CC^n$ is a point near several isolated zeros of $\cF$, i.e., $z^\ast$ approximates a \emph{cluster} of zeros of $\cF$.  The \emph{zero cluster isolation problem} is to compute two closed regions $\smallR$ and $\largeR$ and a positive integer $c$ such that
	\begin{enumerate}
		\item $z^\ast\in \smallR\subseteq \largeR^\circ$, where $\largeR^\circ$ is the interior of $\largeR$ and
		\item the number of zeros of $\cF$ is the same in both $\smallR$ and $\largeR$ and equals $c$.
	\end{enumerate}
	In other words, $\smallR$ encircles a cluster of $c$ zeros of $\cF$, and this cluster of zeros is isolated from the other zeros of $\cF$ by $\largeR\setminus \smallR$.  We also consider the relaxation where $c$ is an upper bound on the number of zeros in $\smallR$ and $\largeR$.
	
	We develop the method of \emph{inflation}, which applies in the square system case ($m=n$) and gives the exact count $c$ when it succeeds.  When inflation fails and in the overdetermined case, we provide a method that yields an upper bound on the size of the cluster.
	
	At a high level, we have the following steps:
	\begin{enumerate}
		\item From the given system $\cF$, find a \emph{nearby} system $\cG$ with a singularity at $z^\ast$, 
		\item compute the structure of the singularity of $\cG$ at $z^\ast$, and
		\item use the relationship between $\cF$ and $\cG$ to infer the location and count of the zeros of $\cF$ near $z^\ast$ from the structure of the singularity of $\cG$ at $z^\ast$.
	\end{enumerate}
	
	The word nearby should only be used in a colloquial and motivational sense since we do not provide a metric for identifying nearness.  We consider both numerical and symbolic perturbations of $\cF$ to generate $\cG$, but we require the final computation to be \emph{certified}.  In other words, as part of their computations, our algorithms not only generate both the integer $c$ and the regions $\smallR$ and $\largeR$, but they also provide a proof of correctness, showing that $\smallR$, $\largeR$, and $c$ have the required properties.
	
	Since all of our constructions and computations pertain to a small neighborhood of one point and tolerate small perturbations of functions in that neighborhood, one may replace analytic functions with \emph{polynomials} as long as there is an effective way to estimate the difference with the original functions.  Hence, we focus on the polynomial case throughout the remainder of the paper.
	
	\subsection{Motivation and contribution}
	Many numeric and symbolic algorithms struggle with computing or approximating zeros of zero-dimensional systems of polynomials that are either singular or clustered.  For some algorithms, however, providing information about the clusters, such as their sizes, locations, and distances from the other zeros, can be used to restore the efficiency of these algorithms \cite{dedieu2001simple,hao2020isolation}.  In addition, data about these clusters can also be used to derive more precise estimates on the algorithmic complexity of algorithms, see, for example, \cite{Sagraloff:2012,Becker2016,Becker2018,BatraSharma:2019}.
	
	Our main contribution is in generalizing the technique dubbed \emph{inflation} and introduced by the first and third authors in \cite{inflation:2021}. Counterintuitively, the inflation procedure transforms a square system with a multiple zero into a square system with the same multiple zero but of \emph{higher} multiplicity.
	
	In~\cite{inflation:2021}, a notion of a \emph{regular zero of order $d$} is defined. In this paper we define a \emph{regular zero of order $d$ and breadth $\kappa$} where: 
	\begin{itemize}
		\item a regular zero of order $d$ corresponds to a regular zero of order $d$ and breadth~$n$,
		\item a regular zero (in the usual sense) is a regular zero of order~$1$. 
	\end{itemize}
	In this new terminology, the original inflation procedure of~\cite{inflation:2021} attempts to create a regular zero of order $2$ from a regular zero of order $2$ and arbitrary breadth. Here we develop \emph{inflation of arbitrary order}, a routine to create a regular zero of order $d$ from a regular zero of order $d$ and arbitrary breadth.
	This turns out to be much more subtle and intricate than the approach in~\cite{inflation:2021}.
	In addition, for systems where inflation cannot be applied directly, we develop new methods to isolate the cluster and provide upper bounds on the size of the cluster.
	
	The shape of isolating regions is dictated by the type of the singularity the input system is close to. Although these regions may in turn be easily bounded by Euclidean balls, this would be an unnecessary relaxation: the region $\smallR$ that we construct (see, for instance, \cref{fig:contours}) is natural and encapsulates the cluster much closer than the ball in which it may be inscribed.     
	
	The \emph{symbolic} procedure of inflation is carried out for $\cG$ in the view of \emph{numerical} applications. Namely, the transformations that we use are applied to a nearby polynomial system $\cF$ with a cluster of zeros. At the end, the effect of the transformations on the difference between $\cF$ and $\cG$ must be small enough to apply the multivariate version of \emph{Rouch\'e's theorem}~\cite[Theorem 2.12]{Rouche:1993}.  We note that our certification step is similar to the certification in \cite{BeckerSagraloff:2017}, but the goals of the papers are different and the use of inflation to regularize the system is one of the novel contributions of the current paper.
	
	We note that the paradigm in which we operate doesn't distinguish between scenarios where there is only one singular zero and scenarios where several simple or singular zeros are tightly clustered. We aim to produce the isolating regions as described in the introduction. We point out that our procedures to construct a nearby system with a singular zero do not work universally. Producing a nearby singular system in a more general setting is the focus of \cite{Mourrain:2020}, for instance. We also assume that an approximation $z^\ast$ is given to us.  There is more focused work on algorithms to approximate a cluster in case of embedding dimension one ~\cite{GiLeSaYa:EmbedDim1} or to restore convergence of Newton's method around a singular solution via deflation~\cite{LVZ}, for example.
	
	Isolating clusters in cases not covered by our technique and finding new algorithms to approximate clusters are worth future exploration.

	
	\subsection{Outline}
	In \Cref{sec:singular}, we consider a square system with a singular zero and, first, introduce necessary transformations to put the system in pre-inflatable shape with a regular zero of breadth $\kappa$ and order $d$, and then inflate in order to isolate the original singular zero.
	In \Cref{sec:nearby}, we demonstrate that the same procedure applied to a nearby system succeeds in isolating a cluster of roots.   
	In \Cref{sec:irregular}, we consider systems that are hard or impossible to put in inflatable shape and show that after symbolic manipulation, it is still possible to isolate a cluster, and the size of the cluster can be bounded from above.
	\Cref{sec:proofs} is devoted to proofs of our results.
	
	\section{Inflation}\label{sec:singular} 
	The first case we consider is a square system which has a singularity at $z^\ast$.  This case is a main step in our general case in \Cref{sec:nearby}  since there we replace the given system with a nearby singular system.  For simplicity, we assume $z^\ast$ is the origin in many of our computations.  Since the point $z^\ast$ is explicitly given or computed as a rational point, no heavy symbolic techniques are needed to perform this translation.
	
	\subsection{Regular breadth-$\bm{\kappa}$ systems of order $\bm{d}$}
	Consider a \emph{graded local order} $>$ on $\CC[x_1,\dots,x_n]$, i.e., the order $>$ respects multiplication and if the total degrees of two exponent vectors $\alpha$ and $\beta$ satisfy $|\alpha|>|\beta|$, then $x^\alpha<x^\beta$.  For a polynomial, we use the phrase \emph{initial term} to denote the largest nonzero monomial under the order $>$, and we use \emph{initial form} to denote the homogeneous polynomial formed from the terms of the polynomial with smallest total degree.  We define the \emph{breadth} $\kappa$ of the polynomial system to be the nullity of its Jacobian.
	
	For an ideal $I=\langle \cF\rangle\subseteq\CC[x_1,\dots,x_n]$, the \emph{standard monomials} are the monomials that do not appear as initial terms of polynomials in $I$.  For each $i$, we define the (local) \emph{Hilbert function} evaluated at $i$, denoted by $h_{\cF}(i)$, to be the number of monomials of total degree $i$ appearing as standard monomials.  The corresponding (local) \emph{Hilbert series} is defined to be $HS_{\cF}(t)=\sum_{i\geq 0} h_{\cF}(i)t^i$.  We note that $h_{\cF}(1)=\kappa$.
	
	\begin{definition}\label{def:regularzero}
		Suppose that $\cP=\{p_1,\dots,p_n\}$ is a square polynomial system $\CC[x_1,\dots,x_n]$ such that the origin is an isolated zero of $\cP$ of breadth $\kappa$.  We say that the origin is a \emph{regular zero of breadth $\kappa$ and order $d$} if the Hilbert series for $\langle\cP\rangle$ at the origin is $(1+t+\dots+t^{d-1})^\kappa$.
	\end{definition}
	
	We note that when the origin is a regular zero of breadth $\kappa$ and order $d$ of a system $\cP$, the multiplicity of the zero at the origin is $d^\kappa$.
	
	Throughout the remainder of this section, we provide \Cref{algo:generalizedInflationcluster}, which converts any square polynomial system into a standardized form, called the pre-inflatable form.
	
	\begin{definition}\label{def:preinflatable}
		Suppose that $\cP=\{p_1,\dots,p_n\}$ is a square polynomial system in $\CC[x_1,\dots,x_n]$ such that the origin is an isolated zero of $\cP$.  We say that $\cP$ is a \emph{$(\kappa,k,\ell)$-pre-inflatable system} if 
		\begin{enumerate}
			\item $\cP$ has breadth $\kappa$ and the kernel of the Jacobian is $\langle e_1,\dots,e_{\kappa}\rangle$, where $e_i$ denotes the $i$-th standard basis vector,
			\item the only terms in $p_1,\dots,p_\kappa$ involving $x_{\kappa+1},\dots,x_n$ have degree greater than $k$, and 
			\item the only terms in $p_{\kappa+1},\dots,p_n$ involving only $x_1,\dots,x_\kappa$ have degree greater than $\ell$.
		\end{enumerate}
	\end{definition}
	
	In the case where our algorithm is applied to a square system with a regular zero of breadth $\kappa$ and order $d$, we prove in \Cref{sec:proofs} that the resulting system is particularly well-structured.  In particular, when the parameters to the pre-inflatable algorithm are $k=\ell=d$, the resulting system is described as in the following theorem:
	
	\ifthenelse{\mainfilecheck{1} > 0}{\begin{thm}\label{thm:HS-for-redular-kappa-d}
			Let $\cG$ be a square system in $n$ variables with a zero at $z^\ast$.  Suppose that $z^\ast$ is a zero of breadth $\kappa$ and order $d$.  Then there is a locally invertible transformation that realizes $z^\ast$ as a regular zero of breadth $\kappa$ and order $d$ at the origin of a polynomial system $\cP=\{p_1,\dots,p_n\}$ which is $(\kappa,d,d)$-pre-inflatable such that 
			\begin{enumerate}
				\item the initial degree of each $p_i$ is equal to $d$ for $1\leq i\leq \kappa$, 
				\item the initial forms of $p_1,\dots,p_\kappa$ do not vanish on the unit sphere in $x_1,\dots,x_\kappa$, and
				\item the initial form of $p_i$ is $x_i$ for $\kappa+1\leq i\leq n$.
			\end{enumerate}
	\end{thm}}
	{\begin{theorem}\label{thm:HS-for-redular-kappa-d}
			Let $\cG$ be a square system in $n$ variables with a zero at $z^\ast$.  Suppose that $z^\ast$ is a zero of breadth $\kappa$ and order $d$.  Then there is a locally invertible transformation that realizes $z^\ast$ as a regular zero of breadth $\kappa$ and order $d$ at the origin of a polynomial system $\cP=\{p_1,\dots,p_n\}$ which is $(\kappa,d,d)$-pre-inflatable such that 
			\begin{enumerate}
				\item the initial degree of each $p_i$ is equal to $d$ for $1\leq i\leq \kappa$, 
				\item the initial forms of $p_1,\dots,p_\kappa$ do not vanish on the unit sphere in $x_1,\dots,x_\kappa$, and
				\item the initial form of $p_i$ is $x_i$ for $\kappa+1\leq i\leq n$.
			\end{enumerate}
	\end{theorem}}

	We observe that when the second condition holds, the initial forms of $p_1,\dots,p_\kappa$ form a regular sequence.  Systems of the form described in \Cref{thm:HS-for-redular-kappa-d} are ideal for applying inflation.
	
	The \emph{inflation operator of order $d$ and breadth $\kappa$} is defined to be
	$$
	S_\kappa^d(x_i)=\begin{cases}
		x_i&1\leq i\leq\kappa\\
		x_i^d&\kappa+1\leq i\leq n
	\end{cases}.
	$$
	The inflation operator in \cite{inflation:2021} is of order $2$ and breadth $\kappa$.
	
	\subsection{Constructing regular zeros}\label{sec:constructingRegularZeros}
	
	Suppose that a given system $\cG$ has a singular zero at the origin of breadth $\kappa$.  We present a sequence of transformations to construct an equivalent system that is $(\kappa,k,\ell)$-pre-inflatable for any given $k,\ell\in\NN$, see \cref{algo:preinflation}. 
	
	First, since $\cG$ has breadth $\kappa$, there is a linear transformation $A:\CC^n\rightarrow\CC^n$ so that the kernel of the Jacobian of $\cA:=\cG\circ A$ is spanned by $e_1,\dots,e_\kappa$, where $e_i$ denotes the $i$-th standard basis vector.  This implies that the linear parts of the polynomials in $\cA$ only involve $x_{\kappa+1},\dots,x_n$.
	
	Second, there is a linear map $B:\CC[x_1,\dots,x_n]^n\rightarrow\CC[x_1,\dots,x_n]^n$ such that $\cB:=B\circ \cA=\{b_1,\dots, b_n\}$ is a square system of polynomials where $b_1,\dots,b_\kappa$ do not have any linear terms while the linear part of $b_i$ is $x_i$ for $i>\kappa$.  The map $B$ can be chosen to implement row reduction on the linear parts of the polynomials of $\cA$.
	
	Next, for the given $k$, there is an invertible \ifthenelse{\mainfilecheck{1} > 0}{}{polynomial} linear map $C_k:\CC[x_1,\dots,x_n]^n\rightarrow\CC[x_1,\dots,x_n]^n$ such that $\cC_k:=C_k\circ \cB=\{c_1,\dots,c_n\}$ is a square system of polynomials with the same properties as $\cB$, and, in addition, in $c_1,\dots,c_\kappa$, the smallest total degree of a term involving $x_{\kappa+1},\dots,x_n$ is greater than $k$.  This transformation can be achieved by using the initial terms of $b_{\kappa+1},\dots,b_n$ to eliminate monomials involving $x_{\kappa+1},\dots,x_n$ of small degree.
	
	Finally, for the given $\ell$, there is an invertible change of variables, denoted by $D_\ell$, such that $\cP_{k,\ell}:=\cC_k\circ D_\ell=\{p_1,\dots,p_n\}$ is a square system of polynomials with the same properties as $\cC_k$ and the smallest degree of a term in $p_{\kappa+1},\dots,p_n$ involving only $x_1,\dots,x_\kappa$ is greater than $\ell$.  This change of variables can be achieved by a sequence of transformations of the form $x_i\mapsto x_i+q_i(x_1,\dots,x_\kappa)$ for some polynomial $q_i$ and the remaining variables are left unchanged.
	\algrenewcommand\algorithmicrequire{\textbf{Input}:}
	\algrenewcommand\algorithmicensure{\textbf{Output}:}
	\begin{algorithm}[ht]
		\caption{Pre-inflation construction}
		\label{algo:preinflation}
		\begin{algorithmic}[1]
			\Require A square polynomial system $\cG$ with a singular zero $z^\ast$ of breadth $\kappa$, and integers $d$ and $\ell$.
			\Ensure A $(\kappa,k,\ell)$-pre-inflatable system whose zero at the origin is of the same multiplicity as $z^\ast$ for $\cG$.
			\State {Apply an affine transformation $A:\CC^n\rightarrow\CC^n$ so that $A(0) = z^\ast$ and the kernel of the Jacobian of $\cA=\cG\circ A$ is spanned by the standard basis vectors $e_1,\dots, e_\kappa$.}\label{step:A} 
			\State {Apply a linear map $B:\CC[x_1,\dots,x_n]^n\rightarrow\CC[x_1,\dots,x_n]^n$ to construct the system $\cB=B\circ\mathcal{A}=\{b_1,\dots,b_n\}$ such that $b_i$ for $i=1,\dots, \kappa$ do not have any linear terms and the linear form of $b_i$ is $x_i$ for $i>\kappa$.}\label{step:B}
			\State {Apply a linear map $C_k:\mathbb{C}[x_1,\dots,x_n]^n\rightarrow\mathbb{C}[x_1,\dots, x_n]^n$ to produce the system $\cC_k=C_k\circ \cB=\{c_1,\dots, c_n\}$ such that the smallest total degree of a term with $x_{\kappa+1},\dots, x_n$ in $c_1,\dots, c_\kappa$ is greater than $k$.}\label{step:C}
			\State {Apply a change of variables $D_\ell$ producing the system $\mathcal{P}_{k,\ell}=\mathcal{C}_k\circ D_\ell=\{p_1,\dots,p_n\}$ such that the smallest total degree of a term in $p_{\kappa+1},\dots, p_n$ with only $x_1,\dots, x_\kappa$ is greater than $\ell$.} \label{step:T}
		\end{algorithmic}
	\end{algorithm}
	
	The property of interest in this series of transformations is the consequence of the \Cref{lem:preinflatable}, which proves that the resulting system is $(\kappa,k,\ell)$-pre-inflatable.  The following example explicitly illustrates this construction:

	\begin{example}\cite[Example 4.1]{OjikaExample}\label{ex:preinflatable}
		Consider the polynomial system
		$$
		\cG=\begin{Bmatrix}
			2x_1+x_2+x_1^2\\
			8x_1+4x_2+x_2^2
		\end{Bmatrix}.
		$$
		This system has a zero at the origin and its Jacobian is $\begin{pmatrix}2&1\\8&4\end{pmatrix}$, which has a one-dimensional nontrivial kernel spanned by $\langle 1,-2\rangle$.  Therefore, this system is breadth-one.  We construct a $(1,3,3)$-pre-inflatable system from $\cG$.
		
		For the linear transform $A$, we use the matrix $\frac{1}{\sqrt{5}}\begin{pmatrix}1&2\\-2&1\end{pmatrix}$, which is the unitary matrix which maps the first standard basis vector to a nonzero element of the kernel of the Jacobian and the second standard basis vector to a vector perpendicular to the kernel.  The resulting system is
		$$
		\cA=\cG\circ A=\begin{Bmatrix}
			\sqrt{5}x_2+\frac{x_1^2}{5}+\frac{4x_1x_2}{5}+\frac{4x_2^2}{5}\\
			4\sqrt{5}x_2+\frac{4x_1^2}{5}-\frac{4x_1x_2}{5}+\frac{x_2^2}{5}
		\end{Bmatrix}.
		$$
		
		Next, row reduction on the linear part of this system, expressed via the matrix $\begin{pmatrix}
			1&-\frac{1}{4}\\0&\frac{1}{4\sqrt{5}},
		\end{pmatrix}$ results in the system
		$$
		\cB=B\circ\cA=\begin{Bmatrix}
			x_1x_2+\frac{3x_2^2}{4}\\
			x_2+\frac{x_1^2}{5\sqrt{5}}-\frac{x_1x_2}{5\sqrt{5}}+\frac{x_2^2}{20\sqrt{5}}
		\end{Bmatrix}.
		$$
		
		Next, the transformation for $k=3$ is the symbolic transformation that uses the initial $x_2$ of the second polynomial to eliminate monomials involving $x_2$ in the first equation.  The transformation is given by the matrix 
		\ifthenelse{\mainfilecheck{1} > 0}{$$C_3=\begin{pmatrix}-5\sqrt{5}&5\sqrt{5}x_1+\frac{15\sqrt{5}x_2}{4}+\frac{x_1^2}{4}+\frac{x_1x_2}{2}-\frac{3x_2^2}{16}\\0&1\end{pmatrix},$$}{$C_3=\begin{pmatrix}-5\sqrt{5}&5\sqrt{5}x_1+\frac{15\sqrt{5}x_2}{4}+\frac{x_1^2}{4}+\frac{x_1x_2}{2}-\frac{3x_2^2}{16}\\0&1\end{pmatrix}$,} which arrives at the system
		$$
		\cC_3=C_3\circ \cB=\begin{Bmatrix}
			x_1^3+\frac{x_1^4}{20\sqrt{5}}+\frac{x_1^3x_2}{20\sqrt{5}}-\frac{x_1^2x_2^2}{8\sqrt{5}}+\frac{x_1x_2^3}{16\sqrt{5}}-\frac{3x_2^4}{320\sqrt{5}}\\[.15cm]
			x_2+\frac{x_1^2}{5\sqrt{5}}-\frac{x_1x_2}{5\sqrt{5}}+\frac{x_2^2}{20\sqrt{5}}
		\end{Bmatrix}.
		$$
		
		Finally, the change of variables for $\ell=3$ absorbs the unwanted terms involving $x_1$ into $x_2$ via the transformation where $x_2$ is replaced by $x_2-\frac{x_1^2}{5\sqrt{5}}-\frac{x_1^3}{125}$.  This results in rather long polynomials, but many of the coefficients are quite small in absolute value.  The resulting system starts with
		$$
		\cP_{3,3}=\cH\circ D_3=\begin{Bmatrix}
			x_1^3+\frac{x_1^4}{20\sqrt{5}}+\frac{x_1^3x_2}{20\sqrt{5}}-\frac{x_1^2x_2^2}{8\sqrt{5}}+\frac{x_1x_2^3}{16\sqrt{5}}-\frac{3x_2^4}{320\sqrt{5}}+\dots\\[.15cm]
			x_2-\frac{x_1x_2}{5\sqrt{5}}+\frac{x_2^2}{20\sqrt{5}}-\frac{x_1^2x_2}{250}+\frac{x_1^4}{500\sqrt{5}}-\frac{x_1^3x_2}{1250\sqrt{5}}+\dots
		\end{Bmatrix}
		$$
	\end{example}
	
	Since the transformations $B$ and $C_k$ are both invertible for any $k$, we observe that $\cA$, $\cB$, and $\cC_k$ all generate the same ideal.  The construction of $\cP_{k,\ell}$ from $\cG$, as in \Cref{ex:preinflatable}, forms a preprocessing step so that the resulting system is pre-inflatable.
	
	Unfortunately, not all pre-inflatable systems can be successfully inflated so that their zeros can be isolated using our techniques.  Only those systems where the $(\kappa,d,d)$-pre-inflatable system has a regular zero of breadth $\kappa$ and order $d$ can be inflated with our techniques.

	
	
	\subsection{Applying inflation}
	
	Suppose that $\cP$ is a $(\kappa,d,d)$-pre-inflatable system where the origin is a regular zero of breadth $\kappa$ and order $d$.  The system $\cP\circ S_\kappa^d$ is then a square system where the initial forms are all of degree $d$ and form a regular sequence, see Section \ref{sec:mainproof}.  Therefore, $\cP\circ S_\kappa^d$ has a zero of multiplicity $d^n$ at the origin.
	
	Let $(\cP\circ S_\kappa^d)_d$ denote the square homogeneous system of degree $d$ consisting of the initial forms of $\cP\circ S_\kappa^d$.  Since this system does not vanish on the unit sphere, let $M$ be a positive lower bound on $\|(\cP\circ S_\kappa^d)_d\|$ over the (Hermitian) unit sphere.
	
	Since all of the terms of $\cP\circ S_\kappa^d-(\cP\circ S_\kappa^d)_d$ are of degree greater than $d$, there is a constant $C>0$ such that for all $\|x\|\leq 1$, $\|\cP\circ S_\kappa^d(x)-(\cP\circ S_\kappa^d)_d(x)\|\leq C\|x\|^{d+1}.$  Then, by applying Rouch\'e's theorem as in \Cref{lem:RoucheVersion} both $(\cP\circ S_\kappa^d)_d$ and $\cP\circ S_\kappa^d$ have $d^n$ zeros in the ball of radius $\varepsilon$.  
	
	While it is straight-forward to observe that the origin is a zero of multiplicity $d^n$ for $\cP\circ S_\kappa^d$, the content of this computation is that there are no additional zeros in the ball of radius $\varepsilon$. The process established is summarized in Algorithm \ref{algo:generalizedInflation}.

	
	\begin{example}\label{ex:preinflatable:continued}
		Continuing \Cref{ex:preinflatable}, the polynomial system $\cP_{3,3}$ is $(1,3,3)$-pre-inflatable, and the origin is a regular zero of breadth $1$ and order $3$.  The inflation step replaces $x_2$ with $x_2^3$.  The resulting system is 
		$$
		\cP_{3,3}\circ S_1^3=\begin{Bmatrix}
			x_1^3+\frac{x_1^4}{20\sqrt{5}}-\frac{x_1^5}{500}-\frac{7x_1^6}{5000\sqrt{5}}+\frac{x_1^3x_2^3}{20\sqrt{5}}-\frac{x_1^7}{10000}+\cdots\\
			x_2^3+\frac{x_1^4}{500\sqrt{5}}-\frac{x_1x_2^3}{5\sqrt{5}}+\frac{x_1^5}{31250}-\frac{x_1^2x_2^3}{250}+\frac{x_1^6}{312500\sqrt{5}}+\cdots
		\end{Bmatrix}.
		$$
		In this example, $(\cP_{3,3}\circ S_1^3)_3=\{x_1^3,x_2^3\}$, and $\|(\cP_{3,3}\circ S_1^3)_3(x)\|\geq \frac{1}{2}$ for $x$ of norm $1$.  We observe that the sum of the absolute values of the noninitial coefficients of $\cP_{3,3}\circ S_\kappa^3$ is $\frac{39599350003}{78125000000\sqrt{5}}+\frac{98831503}{3906250000}\approx 0.251981$.  Since this is less than $\frac{1}{2}$, for any fixed $0<\varepsilon\leq1$ and $\|x\|=\varepsilon$,
		\begin{equation}\label{eq:rouche-condition}
			\|(\cP_{3,3}\circ S_1^3)_3(x)\|>\|\cP_{3,3}\circ S_1^3(x)-(\cP_{3,3}\circ S_1^3)_3(x)\| .
		\end{equation}
		Hence, Rouch\'e's theorem applies and $\cP_{3,3}\circ S_1^3$ has $3^2$ zeros in the ball of radius $\varepsilon$. 
		
		Since the ball containing the $9$ zeros is defined by $|x_1|^2+|x_2|^2\leq\varepsilon$, we can apply the inverse of the changes of variables $A$, $D_3$, $S_1^3$ to compute the following region in the original coordinates, 
		$$
		\frac{1}{5}\left|x_1-2x_2\right|^2+
		\frac{1}{5^{1/3}}\left|\left(2x_1+x_2\right)+\frac{\left(x_1-2x_2\right)^2}{25}+\frac{\left(x_1-2x_2\right)^3}{625}\right|^{\frac{2}{3}}\leq\varepsilon^2,
		$$
		containing the triple zero of $\cF$.  We observe that the inflation map creates a three-to-one cover of the zeros of $\cP_{3,3}\circ S_1^3$ to those of $\cP_{3,3}$, which confirms the root count of $\cF$.
	\end{example}
	
	\ifthenelse{\mainfilecheck{1} > 0}{\begin{thm}\label{thm:main}
			Suppose that $\cG$ is a square polynomial system where $z^\ast$ is a regular zero of breadth $\kappa$ and order $d$ of $\cG$. \Cref{algo:generalizedInflation} produces a region containing $z^\ast$ and no other zeros of $\cG$.  Moreover, the multiplicity of the zero at $z^\ast$ is $d^\kappa$.
	\end{thm}}{\begin{theorem}\label{thm:main}
			Suppose that $\cG$ is a square polynomial system where $z^\ast$ is a regular zero of breadth $\kappa$ and order $d$ of $\cG$. \Cref{algo:generalizedInflation} produces a region containing $z^\ast$ and no other zeros of $\cG$.  Moreover, the multiplicity of the zero at $z^\ast$ is $d^\kappa$.
	\end{theorem}}
	
	\algrenewcommand\algorithmicrequire{\textbf{Input}:}
	\algrenewcommand\algorithmicensure{\textbf{Output}:}
	\begin{algorithm}[ht]
		\caption{Generalized inflation for isolating singular zeros}
		\label{algo:generalizedInflation}
		\begin{algorithmic}[1]
			\Require A square polynomial system $\cG$ with a regular zero $z^\ast$ of breadth $\kappa$ and order $d$.
			\Ensure A region $\largeR$ containing the zero $z^\ast$ and no other zeros of~$\cG$.
			\State {Apply \cref{algo:preinflation} to $\cG$ to construct a $(\kappa,d,d)$-preinflatable system $\cP_{d,d}$.} \label{step:preinflation}
			\State {Using the inflation operator $S_\kappa^d$, construct the system $\mathcal{P}_{d,d}\circ S_\kappa^d$, and find $\varepsilon$ so that $\mathcal{P}_{d,d}\circ S_\kappa^d$ has $d^n$ zeros in the ball of radius $\varepsilon$.}
			\State {Apply the inverse of $A\circ D_\ell \circ S_\kappa^d$ to the ball of radius $\varepsilon$ to get the isolating region $\largeR$.}
		\end{algorithmic}
	\end{algorithm}
	
	We note that that when considering one (exact) singular zero, we produce only the large region $\largeR$ since the small region $\smallR$ can be taken to be trivial, i.e., $\smallR = \{z^\ast\}$.
	
	
	
	
	

	\section{Clusters of Zeros}\label{sec:nearby} 
	In the intended application of our approach, we do not expect to be given a system that has a multiple zero, as explored in \Cref{sec:singular}.  Instead, we expect to be given a system that has a cluster of zeros, each with multiplicity one.  Suppose that $\cF$ is a square system of polynomials and $z^\ast$ approximates the center of a cluster of zeros of $\cF$.  Our approach is to find a nearby singular system and use that system to inform about the zeros of $\cF$.
	
	\subsection{Isolating clusters}\label{sec:isolating}
	
	Suppose that system $\cG$ has a (singular) zero at $z^\ast$ whose coefficients are close to those in $\cF$.  Suppose also there exist invertible maps $T:\CC^n\rightarrow\CC^n$ and $U:\CC[x_1,\dots,x_n]^n\rightarrow\CC[x_1,\dots,x_n]^n$ such that the origin is a regular zero of breadth $\kappa$ and order $d$ of $U\circ \cG\circ T$.  One candidate for $U$ and $T$ is presented in \Cref{sec:constructingRegularZeros}.  We then apply these maps and inflation to the original system to get the system $U\circ\cF\circ T\circ S_\kappa^d$.  Let $(U\circ\cF\circ T\circ S_\kappa^d)_d$ denote the homogeneous part of this system of degree $d$.  Similarly, we write $(U\circ\cF\circ T\circ S_\kappa^d)_{>d}$ and $(U\circ\cF\circ T\circ S_\kappa^d)_{<d}$ for the terms greater than or less than $d$.
	
	Let $M$ be a positive lower bound on $\|(U\circ\cF\circ T\circ S_\kappa^d)_d\|$ over the (Hermitian) unit sphere.  Since all the terms of $(U\circ\cF\circ T\circ S_\kappa^d)_{>d}$ are of degree greater than $d$, there is a constant $M_1>0$ such that for all $\|x\|\leq 1$, $\|(U\circ\cF\circ T\circ S_\kappa^d)_{>d}\|\leq M_1\|x\|^{d+1}$.  Similarly, since all terms of $(U\circ\cF\circ T\circ S_\kappa^d)_{<d}$ have degree less than $d$, there is a constant $M_2>0$ such that for all $\|x\|\leq 1$, $\|(U\circ\cF\circ T\circ S_\kappa^d)_{<d}\|<M_2$.  If $\left(\frac{2M_2}{M}\right)^{1/d}<\frac{M}{2M_1}$, then for any $\varepsilon$ between $\smallEpsilon=\left(\frac{2M_2}{M}\right)^{1/d}$ and $\largeEpsilon=\frac{M}{2M_1}$, $\|(U\circ\cF\circ T\circ S_\kappa^d)_d(x)\|$ dominates the other parts of $U\circ\cF\circ T\circ S_\kappa^d$ and, by Rouch\'e's theorem, see Lemma \ref{lem:RoucheVersion}, $U\circ\cF\circ T\circ S_\kappa^d$ and $(U\circ\cF\circ T\circ S_\kappa^d)_d$ 
	have the same number of zeros in the ball of radius $\varepsilon$.  The smaller region $\smallR$ corresponds to the lower bound $\smallEpsilon$ and the larger region $\largeR$ corresponds to the upper bound $\largeEpsilon$ on $\varepsilon$.


	\begin{example}
		Consider the polynomial system
		$$
		\cF=\begin{Bmatrix}
			2x_1+x_2+x_1^2+0.001\\
			8x_1+4x_2+x_2^2+0.001
		\end{Bmatrix}
		$$
		with approximate zero $z^\ast=(-0.0001,-0.0001)$.  This system is a perturbation of our running example from \cref{ex:preinflatable}. The three zeros in the cluster are approximately $(-0.043 - 0.082 i, 0.091 + 0.158 i)$, 
		$(-0.043 + 0.082 i, 0.091 + 0.158 i)$, and $(0.086, -0.181)$, and $z^\ast$ approximates their average.
		
		First, we shift the system $\cF$ so that $z^\ast$ is at the origin.  The resulting system is very close to system $\cG$ from \Cref{ex:preinflatable}.  After applying the same transformations from Examples \ref{ex:preinflatable} and \ref{ex:preinflatable:continued}, the resulting system is (after rounding) is
		$$
		\begin{Bmatrix}
			-0.0084+0.0013x_1+0.000078x_1^2+x_1^3+0.0016x_2^3+\dots\\
			-0.000022+0.00002x_1+0.00000089x_1^2+0.00000008x_1^3+x_2^3+\dots
		\end{Bmatrix}.
		$$
		In this case, the cubic part of the system is bounded from below on the unit circle by $0.4984$.  On the other hand, the sum of the coefficients of degree less than $3$ is greater than $0.009757$, which can be used for $M_2$.  Finally, the sum of the coefficients of degree greater than $3$ is less than $0.2746$, which can be used for $M_1$.  Therefore, we may choose $\largeEpsilon=0.9075$ and $\smallEpsilon=0.3396$, as illustrated in \cref{fig:contours} in the original domain.
		
		The only other zero of the system is approximately equal to $(4,8)$ and is far away from all depicted isolating regions. We also note that the regions are not convex and the boundaries of the regions are only piecewise smooth.
		
		\begin{figure}[ht]
			\includegraphics[width=0.5\linewidth]{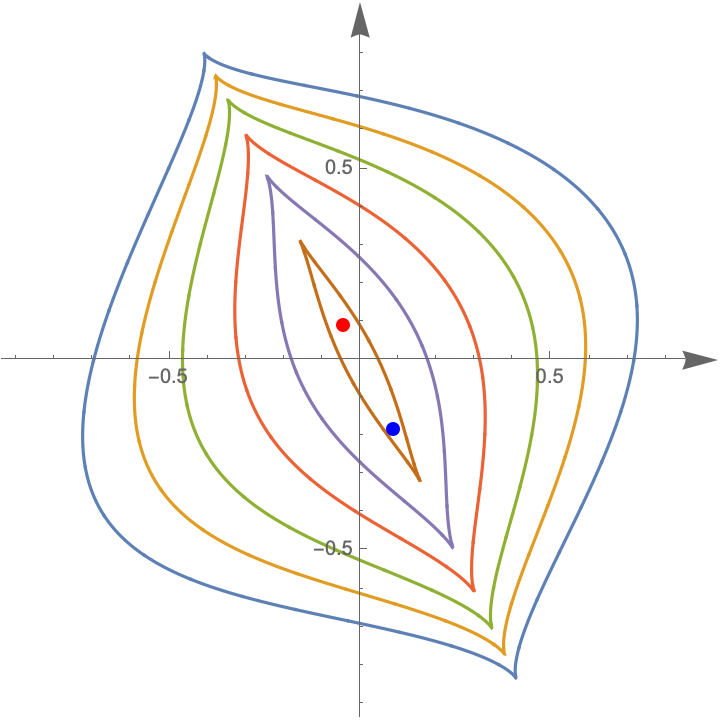}
			\caption{Contours of isolating regions for a cluster of zeros of \begin{math}\cF\end{math}. The inner- and outer-most contours are the boundaries of $\smallR$ and $\largeR$, the smallest and largest isolating regions our method produces.  The {\color{red}(red)} point in the second quadrant depicts the real part of two conjugate nonreal zeros.} 
			\label{fig:contours}
		\end{figure}
		
	\end{example}
	
	\begin{algorithm}[ht]
		\caption{Generalized inflation for isolating clusters of zeros}
		\label{algo:generalizedInflationcluster}
		\begin{algorithmic}[1]
			\Require A square polynomial system $\cG$ with a cluster of zeros near $z^\ast$ and $d\in\NN$.
			\Ensure A pair of regions $\largeR$ and $\smallR$ containing the cluster and no other zeros of $\cG$ such that $\smallR\subseteq\largeR^\circ$.
			\State {Construct a singular system $\cG$ close to the given system.}
			\State {Apply \cref{algo:preinflation} with parameters $k=\ell=d$ to $\cG$ and collect the two invertible maps $U$ and $T$ applied to $\cG$ as $U\circ\cG\circ T$.}
			\State {Compute $U\circ \cF\circ T\circ S_\kappa^d$.}
			\State {Compute a lower bound $M$ on $(U\circ \cF\circ T\circ S_\kappa^d)_d$ on the (Hermitian) unit sphere.}
			\State {Compute an upper bound $M_1$ on $(U\circ \cF\circ T\circ S_\kappa^d)_{>d}/\|x\|^{d+1}$ on the unit disk.}
			\State {Compute an upper bound $M_2$ on $(U\circ \cF\circ T\circ S_\kappa^d)_{<d}$ on the unit disk.}
			\State {Compute $\smallEpsilon=\left(\frac{2M_2}{M}\right)^{1/d}$ and $\largeEpsilon=\frac{M}{2M_1}$.}
			\If {$\smallEpsilon<\largeEpsilon$}
			\State{Apply the inverse of $T\circ S_\kappa^d$ to the balls of radii $\smallEpsilon$ and $\largeEpsilon$ to get the isolating regions $\smallR$ and $\largeR$.}
			\EndIf
		\end{algorithmic}
	\end{algorithm}
	
	\begin{remark}
		One approach to compute $M$ is based on sum-of-squares computations as in  \cite{inflation:2021}.  For the bounds $M_1$ and $M_2$, one way to get these bounds is to sum the absolute values of the coefficients appearing in the appropriate systems.
	\end{remark}
	
	\ifthenelse{\mainfilecheck{1} > 0}{\begin{thm}\label{thm:main2}
			Suppose that $\cF$ is a square polynomial system where $z^\ast$ approximates a cluster of zeros.  If \Cref{algo:generalizedInflationcluster} succeeds, then it produces a pair of regions $\smallR$ and $\largeR$ containing $z^\ast$ and the cluster of zeros such that $\smallR\subseteq\largeR^\circ$.  Moreover, the number of zeros in the cluster is $d^\kappa$.
	\end{thm}}{\begin{theorem}\label{thm:main2}
			Suppose that $\cF$ is a square polynomial system where $z^\ast$ approximates a cluster of zeros.  If \Cref{algo:generalizedInflationcluster} succeeds, then it produces a pair of regions $\smallR$ and $\largeR$ containing $z^\ast$ and the cluster of zeros such that $\smallR\subseteq\largeR^\circ$.  Moreover, the number of zeros in the cluster is $d^\kappa$.
	\end{theorem}}
	
	\subsection{Constructing a singular system}
	
	In order to complete the steps outlined in \Cref{sec:isolating}, we need to be able to construct an appropriate singular system $\cG$.  One way to construct such a system is outlined in \cite[Section 2.1]{inflation:2021} via the singular value decomposition of the Jacobian $D\cF(z^\ast)$.  This construction also provides $\kappa$ as a count of the number of small singular values of the Jacobian.  
	
	For any $d>0$, we may apply \Cref{algo:generalizedInflation} to $\cG$ to construct a $(\kappa,d,d)$-pre-inflatable system.  It is unlikely that the resulting system has the origin as a regular zero of breadth $\kappa$ and order $d$.
	
	Even though the polynomial system resulting from \Cref{algo:generalizedInflation} might not be amenable to inflation itself, the constructed transformations, when applied to $\cF$ as in \Cref{sec:isolating}, may succeed in isolating the cluster of $\cF$.  We have experimental evidence that applying these transformations will be successful when extra terms in $p_1,\dots,p_\kappa$ have small coefficients.
	
	\begin{remark}
		The value of $d$ in the construction of the pre-inflatable system may either be given or guessed through the computations.  In particular, we may apply \Cref{algo:generalizedInflation} with many different values of $d$ until the degree-$d$ homogeneous part of the resulting system has enough terms with coefficients larger than some tolerance as to not vanish on the unit sphere.
	\end{remark}
	
	\section{Irregular systems}\label{sec:irregular}
	Even when the approach in \Cref{sec:singular} fails for singular systems, we present ways to isolate the cluster and estimate its size.  Three instances where the approach in \Cref{sec:singular} may fail are when the origin is a regular zero of breadth $\kappa$ and order $d$, when the initial forms of the $(\kappa,d,d)$-pre-inflatable system vanish on the unit sphere, and when the initial system is not square.
	
	\subsection{Uneven inflation}
	
	The structure of the polynomial system in \cref{thm:HS-for-redular-kappa-d} are designed so that we know the structure of the system after inflation, see \cref{sec:mainproof}.  In particular, several of the steps in the construction of a pre-inflatable system are designed to control which terms appear in the initial form of the system after inflation.
	
	When the initial forms of a polynomial system do not vanish on the unit sphere, but they do not have the appropriate degrees, we may apply an inflation operator that changes the degree of each variable individually.  To illustrate this, consider the following motivating example:
	
	\begin{example}
		Consider the following family of polynomial systems, where $a$ is a parameter:
		$$
		\cG=\begin{Bmatrix}
			x_1\\
			x_2^2+ax_3^2+x_3^4\\
			x_3^3
		\end{Bmatrix}.
		$$
		An initial attempt might be to inflate by replacing $x_1$, $x_2$, and $x_3$ by $x_1^6$, $x_2^3$, and $x_3^2$, respectively.  Unfortunately, after this inflation step, the resulting system is
		$$
		\begin{Bmatrix}
			x_1^6\\
			ax_3^4+x_2^6+x_3^8\\
			x_3^6
		\end{Bmatrix}.
		$$
		We cannot apply our approaches unless $a=0$, in which case the initial forms are all of degree $6$ and do not vanish on the unit sphere.  
		
		When $a=0$, the inflated system has a zero of multiplicity $6^3=216$ at the origin and Rouch\'e's theorem can be applied to isolate these zeros.  Moreover, since the inflation map is $36$-to-one, this region isolates the $6$ solutions of the original system.
		
		When $a\not=0$, it is impossible to choose an inflation map so that the initial forms are all of the same degree.  In this case, the inflation approach fails and we must consider alternate methods.
	\end{example}
	
	For a general singular system $\cG$ where zero is not a regular zero, suppose that it is possible to replace each variable by a power so that all the initial terms of the resulting system have the same degree.  In this case, Rouch\'e's theorem, see \cref{lem:RoucheVersion} can be applied to isolate the cluster.  For this approach to succeed, it is usually important that the initial forms of the system $\cG$ have some structure and that problem-specific higher degree terms have a zero coefficient.
	
	\subsection{Upper bounds}
	
	One may attempt a symbolic transformation that leads to a system where Rouch\'e's theorem applies, see \cref{lem:RoucheVersion}.  Given a singular system $\cG$ of $m$ functions in $n$ unknowns with $m\geq n$ with an isolated zero at the origin, there is an $n\times m$ matrix $T$ such that the initial forms of the polynomials in $T\cG$ do not vanish on the unit sphere.  Even in the case $m=n$, it is possible to find a suitable $T$ that is invertible in a neighborhood of the singularity at the origin. Therefore, the multiplicity of the origin as a zero of $T\cG$ is only an upper bound of the multiplicity of $\cG$.
	
	One popular ``rewriting'' method is to derive $T$ from a local Gr\"obner basis computation.  In particular, we choose $n$ elements whose initial terms are pure powers from the Gr\"obner basis.  This process also applies to the overdetermined case because we are choosing only $n$ elements from the Gr\"obner basis regardless of the number of equations in $\cF$.  We illustrate this method in the following example:
	
	\begin{example}
		Consider the following singular polynomial system
		$$
		\cG=\begin{Bmatrix}
			x_1x_2-x_3^3\\[.05cm]
			x_2x_3-x_1^3\\[.05cm]
			x_1x_3-x_2^3
		\end{Bmatrix}.
		$$
		The initial forms simultaneously vanish on the unit sphere in the coordinate directions.  From a local Gr\"obner basis calculation, there are three elements in the basis whose initial terms are pure powers:
		$$
		\begin{Bmatrix}
			x_2^4-x_3^4\\[.05cm]
			x_1^4-x_2^4\\[.05cm]
			x_3^5-x_1^3x_2^3
		\end{Bmatrix}.
		$$
		Therefore, we can find a system of polynomials in the ideal generated by $\cG$ that have the same degree for their initial forms.  In particular, we have the elements
		$$
		\cP=\begin{Bmatrix}
			x_2^5-x_2x_3^4\\[.05cm]
			x_1^5-x_1x_2^4\\[.05cm]
			x_3^5-x_1^3x_2^3
		\end{Bmatrix}.
		$$
		This system can be obtained by multiplying the equations of $\cG$ by the following matrix, derived from a local Gr\"obner basis calculation:
		$$T = \begin{pmatrix}x_2x_3&0&-x_2^2\\[.05cm]0&-x_1^2&x_1x_2\\[.05cm]-x_3^2&x_2^3&x_2x_3\end{pmatrix}.$$ 
		
		Since the initial forms do not vanish on the unit sphere, we can find a lower bound $M$ for $\|\cP_5\|$ on the unit sphere.  Then, by following the approach of \Cref{sec:singular}, we find a region $\largeR$ that isolates the singularity at the origin.  In this case, the singularity has multiplicity at most $4\cdot 4\cdot 5=80$ while the true multiplicity is $11$.
	\end{example}
	
	In the cluster case, i.e., when $\cF$ is given with $z^\ast$ approximating a cluster of zero, a suitable $T$ can be found by executing the steps of a Gr\"obner basis computation while dropping terms with small coefficients to construct $\cG$.  As long as $T\cF-\cG$ is sufficiently small, then Rouch\'e's theorem can be used.
	
	Since the multiplicity of $z^\ast$ may increase when $\cF$ transforms into $\cG$, this increase also applies to the size of the corresponding cluster of $T\cF$. Thus, this process may not provide the exact size of the cluster, but an upper bound of it.
	
	\section{Proofs}\label{sec:proofs}
	We provide proofs for several of the stated facts in the paper.  
	
	\subsection{Pre-inflatable form}
	We prove that the procedure described in \cref{sec:constructingRegularZeros} and \cref{algo:preinflation} produces a $(\kappa,k,\ell)$-pre-inflatable system for any square polynomial system of breadth $\kappa$.
	
	\begin{lemma}\label{lem:preinflatable}
		Let $\cG$ be a square polynomial system with a singular zero at $z^\ast$ of breadth $\kappa$.  The result of \cref{algo:preinflation} with parameters $k$ and $\ell$ on $\cG$ is a $(\kappa,k,\ell)$-pre-inflatable system of polynomials with a zero at the origin whose multiplicity is the same as the multiplicity of $z^\ast$ for $\cG$.
	\end{lemma}
	\begin{proof}
		We first show that the multiplicity of the origin and $z^\ast$ are the same for the input and output systems.  The first step of the algorithm is an invertible affine transformation on the domain, and such transformations do not change the multiplicity of a zero.  The second and third steps replace the system with a new system that generates the same ideal, hence the multiplicity does not change.  Finally, the last step uses transformations of the form $x_i\mapsto x_i+q_i(x_1,\dots,x_\kappa)$, which preserve leading forms of all polynomials in the ideal, and, hence preserve the Hilbert series and multiplicity.
		
		Now, we prove that the final system is $(\kappa,k,\ell)$-pre-inflatable.  By the discussion above, the breadth of the system does not change under the steps of \cref{algo:preinflation}, therefore, the final system has the correct breadth.  In addition, the affine transformation in the first step rotates the domain so that the resulting Jacobian has the correct kernel.  The second and third steps do not change the kernel of the Jacobian, and the last step maintains the initial terms, so the Jacobian is also preserved.  The first and second steps prepare the initial terms of the last $n-\kappa$ polynomials via standard linear algebra, and these initial terms are not changed in the last two steps.
		
		Finally, the third and fourth steps can be broken down into a sequence of cancellation steps, each of which remove a term of low degree and replace it with terms of higher degree.  Through induction, all of the desired terms have coefficient zero.  Therefore, the resulting system is in $(\kappa,k,\ell)$-pre-inflatable form.
	\end{proof}
	
	This construction explicitly leads to the following corollary, which proves one of the conditions in \cref{thm:HS-for-redular-kappa-d}.
	
	\begin{corollary}
		Let $\cG$ be a square system in $n$ variables with a zero at $z^\ast$.  Suppose that $z^\ast$ is a zero of breadth $\kappa$ and order $d$.  \cref{algo:preinflation} with parameters $k=\ell=d$ applied to this system results in a $(\kappa,d,d)$-pre-inflatable system such that the initial form of $p_i$ is $x_i$ for $\kappa+1\leq i\leq n$.
	\end{corollary}
	
	\subsection{Regular zero form}\label{sec:mainproof}
	The following proof is an algorithmic proof of the remaining two conditions in \cref{thm:HS-for-redular-kappa-d}.  It provides a construction of an analytic change of variables that transforms a system with a regular zero of breadth $\kappa$ and order $d$ into one of the desired form.
	
	Before beginning the proof, we introduce some notation and a fact about Hilbert series.  For series $A(t)=\sum_{i\geq 0}a(i)t^i$ and $B(t)=\sum_{i\geq 0}b(i)t^i$, we let $A(t)\geq B(t)$ if $a(i)\geq b(i)$ for all $i$.  In addition, we consider the following lemma for the proof:
	\begin{lemma}\cite[Lemma 1]{Froberg:1985}\label{lem:FrobergLemma1}
		Consider $\cF=\{f_1,\dots, f_n\}$ and $\cG=\{g_1,\dots, g_n\}$ where $f_1,\dots,f_n,g_1,\dots,g_n$ are homogeneous and $g_1,\dots,g_n$ are generic.  If $\deg g_i=\deg f_i$, then, $HS_\cG(t)\leq HS_\cF(t)$.
	\end{lemma}
	Moreover, generic homogeneous forms of degree $d$ form a regular sequence, and the Hilbert series of a regular sequence is $\frac{(1-t^d)^n}{(1-t)^n}$.

	\begin{lemma}
		Let $\cG$ be a square system in $n$ variables with a zero at $z^\ast$.  Suppose that $z^\ast$ is a zero of breadth $\kappa$ and order $d$.  \cref{algo:preinflation} with parameters $k=\ell=d$ applied to this system results in a $(\kappa,d,d)$-pre-inflatable system such that the initial degree of each $p_i$ is equal to $d$ for $1\leq i\leq\kappa$.
	\end{lemma}
	
	\begin{proof}
		In order to prove this lemma, we begin with the transformation from \Cref{sec:constructingRegularZeros} that transforms the original system $\cG$ into the system $\cP_{d,d}=\{p_1,\dots,p_n\}$ that is a $(\kappa,d,d)$-pre-inflatable system.  As mentioned in the proof of \cref{lem:preinflatable}, the Hilbert series is unchanged under this transformation.
		
		From the definition of a pre-inflatable system, in $\{p_1,\dots,p_k\}$ the only monomials of degree at most $d$ which appear in $p_1,\dots,p_\kappa$ involve only the variables $x_1,\dots,x_\kappa$.  On the other hand, since the initial term of $g_i$ for $i>\kappa$ is $x_i$, it follows that no monomial involving any of $x_{\kappa+1},\dots,x_n$ can appear as a standard monomial.  In addition, the coefficients in the local Hilbert series for $\{1,\dots,t^{d-1}\}$ are the number of monomials in $\kappa$ variables.  This implies that $p_1,\dots,p_\kappa$ cannot have any monomials of degree less than $d$.
		
		We now prove that the initial degree of $p_1,\dots,p_\kappa$ must be $d$.  Suppose that $p\in\langle p_1,\dots,p_n\rangle$ such that the initial term of $p$ is of degree $d$ and involves only $x_1,\dots,x_\kappa$.  Briefly, we write $p=\sum q_ip_i$.  Since, by construction, the monomials in $p_{\kappa+1},\dots,p_n$ involving only $x_1,\dots,x_\kappa$ must have degree larger than $d$, it follows that the initial term of $p$ does not appear in any $q_ip_i$ for $i>\kappa$.  On the other hand, since the initial degree of $p_1,\dots,p_\kappa$ is at least $d$, it must be that the initial term of $p$ is an initial term of an element of $\langle (p_1)_d,\dots,(p_\kappa)_d\rangle$, where $(p_i)_d$ denotes the homogeneous part of $p_i$ of degree $d$.  
		
		This observation implies that the standard monomials of $\langle \cP_{d,d}\rangle$ of degree $d$ are the same as the standard monomials of degree $d$ that only involve $x_1,\dots,x_\kappa$ of $\langle (p_1)_d,\dots,(p_\kappa)_d\rangle$.  Suppose that $(p_i)_d=0$ for $\ell$ values of $i$, where $1\leq i\leq \kappa$.  Then, by Lemma \ref{lem:FrobergLemma1},
		the coefficient of $t^d$ in the Hilbert series of $\langle (p_1)_d,\dots,(p_\kappa)_d\rangle$ is at least the coefficient of $t$ in $\frac{(1-t^d)^{\kappa-\ell}}{(1-t)^\kappa}.$  If $\ell>0$, then this coefficient is (strictly) greater than the corresponding coefficient in $(1+t+\dots+t^{d-1})^\kappa$, a contradiction.  Hence, $\ell=0$ and the initial degree of each of $p_1,\dots,p_\kappa$ must be $d$.
	\end{proof}

	\begin{lemma}\label{lem:regularsequence}
		Let $\cG$ be a square system in $n$ variables with a zero at $z^\ast$.  Suppose that $z^\ast$ is a zero of breadth $\kappa$ and order $d$.  \cref{algo:preinflation} with parameters $k=\ell=d$ applied to this system results in a $(\kappa,d,d)$-pre-inflatable system such that the initial forms of $p_1,\dots,p_\kappa$ do not vanish on the unit sphere in $x_1,\dots,x_\kappa$.
	\end{lemma}

	\begin{proof}
		Suppose that $(p_1)_d,\dots,(p_\kappa)_d$ do not form a regular sequence.  Let $r$ be the smallest degree where there exists $1<j\leq \kappa$ and homogeneous polynomials $m_1,\dots,m_j$ of degree $r-d$ such that $\sum_{i=1}^j m_i(p_i)_d=0$ and $m_j\not\in\langle (p_1)_d,\dots,(p_{j-1})_d\rangle$.  For all degrees $k$ less than $i$ and $1\leq \ell\leq \kappa$, the multiplication map 
		\begin{equation*}
			\left(k[x_1,\dots,x_n]/\langle (p_1)_d,\dots,(p_{\ell-1})_d\rangle\right)_{k-d}\xrightarrow
			{(p_{\ell-1})_d}\left(k[x_1,\dots,x_n]/\langle (p_1)_d,\dots,(p_{\ell})_d\right)_k
		\end{equation*}
		is injective.  Hence, the coefficient of $t^k$ in the Hilbert series for $\{(p_1)_d,\dots,(p_\kappa)_d\}$ agrees with the corresponding coefficient for a regular sequence.  In dimension $r$, this map is not always injective, so the coefficient of $t^r$ in the Hilbert series for $\{(p_1)_d,\dots,(p_\kappa)_d\}$ is larger than the coefficient of $t^r$ for a regular sequence.
		
		We now show that this also implies that the number of standard monomials of $\langle\cP_{d,d}\rangle$ in dimension $r$ contradicts the assumption on the Hilbert series.  Let $p\in\langle p_1,\dots,p_\kappa\rangle$ such that the initial degree of $p$ is $r$ and the initial form of $p$ is not in $\langle (p_1)_d,\dots,(p_{\kappa})_d\rangle$.  Since $p\in\langle p_1,\dots,p_\kappa\rangle$, $p=\sum q_ip_i$ for some polynomials $q_i$.  Suppose that the $q_i$'s have been chosen so that the minimum initial degree of $q_ip_i$ is maximized.  Let $m$ be this initial degree.  Moreover, assume that the $q_i$'s have been chosen so that the largest index where the initial degree of $q_ip_i$ is $m$ is minimized.  Let this index be $\ell$.
		
		Let $(q_i)_{m-d}$ denote the degree $m-d$ homogeneous part of $q_i$.  Since the initial degree of $q_ip_i$ is at least $m$, either $(q_i)_{m-d}=0$ or $(q_i)_{m-d}$ is the initial form of $q_i$.  In addition, $\sum_{i=1}^\ell (q_i)_{m-d}(p_i)_d=0$ since otherwise, this would be the initial form of $p$ and would also be in $\langle (p_1)_d,\dots,(p_{\kappa})_d\rangle$.  Moreover, the sum is not a sum of $0$'s since $(q_\ell)_{m-d}\not=0$.  Therefore, $m<r$ and so, by the assumption on $r$, $(q_\ell)_{m-d}\in\langle (p_1)_d,\dots,(p_{\ell-1})_d\rangle$.  Therefore, there exist homogeneous polynomials $s_1,\dots,s_{\ell-1}$ which are either $0$ or of degree $m-2d$ such that $(q_\ell)_{m-d}=\sum_{i=1}^{\ell-1} s_i(p_i)_d$.  Then,
		\ifthenelse{\mainfilecheck{1} > 0}{
			\begin{align*}
				\sum_{i=1}^\kappa q_ip_i&=\sum_{i=1}^{\ell-1} q_ip_i+q_\ell p_\ell+\sum_{i=\ell+1}^\kappa q_ip_i\\
				&=\sum_{i=1}^{\ell-1} q_ip_i+((q_\ell)_{m-d}+(q_\ell-(q_\ell)_{m-d}))p_\ell+\sum_{i=\ell+1}^\kappa q_ip_i\\
				&=\sum_{i=1}^{\ell-1} q_ip_i+\sum_{i=1}^{\ell-1} s_i(p_i)_dp_\ell+
				(q_\ell-(q_\ell)_{m-d})p_\ell+\sum_{i=\ell+1}^\kappa q_ip_i\\
				&=\sum_{i=1}^{\ell-1} q_ip_i+\sum_{i=1}^{\ell-1} s_i(p_i-(p_i-(p_i)_d))p_\ell+(q_\ell-(q_\ell)_{m-d})p_\ell+\sum_{i=\ell+1}^\kappa q_ip_i\\
				&=\sum_{i=1}^{\ell-1}(q_i+s_ip_\ell)p_i+\left(\sum_{i=1}^{\ell-1}s_i((p_i)_d-p_i)+(q_\ell-(q_\ell)_{m-d})\right)p_\ell+\sum_{i=\ell+1}^\kappa q_ip_i.
		\end{align*}}{\begin{align*}
				\sum_{i=1}^\kappa q_ip_i&=\sum_{i=1}^{\ell-1} q_ip_i+q_\ell p_\ell+\sum_{i=\ell+1}^\kappa q_ip_i\\
				&=\sum_{i=1}^{\ell-1} q_ip_i+((q_\ell)_{m-d}+(q_\ell-(q_\ell)_{m-d}))p_\ell+\sum_{i=\ell+1}^\kappa q_ip_i\\
				&=\sum_{i=1}^{\ell-1} q_ip_i+\sum_{i=1}^{\ell-1} s_i(p_i)_dp_\ell+
				(q_\ell-(q_\ell)_{m-d})p_\ell+\sum_{i=\ell+1}^\kappa q_ip_i\\
				&=\sum_{i=1}^{\ell-1} q_ip_i+\sum_{i=1}^{\ell-1} s_i(p_i-(p_i-(p_i)_d))p_\ell+(q_\ell-(q_\ell)_{m-d})p_\ell+\sum_{i=\ell+1}^\kappa q_ip_i\\
				&=\sum_{i=1}^{\ell-1}(q_i+s_ip_\ell)p_i\\
				&\quad+\left(\sum_{i=1}^{\ell-1}s_i((p_i)_d-p_i)+(q_\ell-(q_\ell)_{m-d})\right)p_\ell+\sum_{i=\ell+1}^\kappa q_ip_i.
		\end{align*}}
		The initial degree of $(p_i)_d-p_i$ is greater than $d$ and that of $q_\ell-(q_\ell)_{m-d}$ is greater than $m-d$ as well. We see that this violates the assumptions on $m$ and $\ell$.  In other words, either the minimum initial degree of a summand is larger or there are fewer terms that attain the degree $m$.  Hence, $(p_1)_d,\dots,(p_\kappa)_d$ form a regular sequence and only have finitely many common zeros in $\kappa$-dimensional affine space.  Therefore, they cannot vanish on the unit sphere $x_1,\dots,x_\kappa$, as, by homogeneity, this would imply that they vanish on a line.
	\end{proof}
	
	The proof of \cref{lem:regularsequence} implies that if the initial forms of $p_1,\dots,p_\kappa$ are a regular sequence, then the initial forms of $\langle P_{d,d}\rangle$ are the same as the forms in $\langle(p_1)_d,\dots,(p_\kappa)_d\rangle$.  Moreover, we can also conclude that if the Hilbert series for $\langle(p_1)_d,\dots,(p_\kappa)_d\rangle$ is $\frac{(1-t^d)^\kappa}{(1-t)^\kappa}$, then $(p_1)_d,\dots,(p_\kappa)_d$ form a regular sequence.
	
	\subsection{Application of Rouch\'e's theorem}
	Finally, we prove the consequence of Rouch\'e's theorem that we use to certify our algorithms.
	\begin{lemma}\label{lem:RoucheVersion}
		Let $\cP$ be a square polynomial system and $\cQ$ be a square homogeneous polynomial system of degree $d$.  Let $\bS_\varepsilon$ denote the $n$-dimensional (Hermitian) unit sphere of radius $\varepsilon$.  Suppose that
		\begin{enumerate}
			\item There is a positive constant $M$ such that $\min\{\|Q(x)\|:x\in\bS_1\}\geq M$ and
			\item There are constants $M_1$ and $M_2$ and a decomposition $\cP=\cP_1+\cP_2+\cQ$ such that for all $\varepsilon\leq 1$
			\begin{enumerate}
				\item $\max\{\|\cP_1(x)\|:x\in\bS_\varepsilon\}\leq M_1\varepsilon^{d+1}$.
				\item $\max\{\|\cP_2(x)\|:\|x\|\leq 1\}\leq M_2$
			\end{enumerate}
		\end{enumerate}
		If $\left(\frac{2M_2}{M}\right)^{1/d}<\frac{M}{2M_1}$, then for any $\varepsilon\in\left[\left(\frac{2M_2}{M}\right)^{1/d},\frac{M}{2M_1}\right]$, $\cP$ has $d^n$ zeros in the ball of radius $\varepsilon$.
	\end{lemma}
	\begin{proof}
		The first condition implies that $Q$ has no zeros on the unit sphere, so all of its $d^n$ zeros are at the origin.  For $x$ satisfying the given conditions, 
		\ifthenelse{\mainfilecheck{1} > 0}{$$\|\cP(x)-\cQ(x)\|\leq \|\cP_1(x)\|+\|\cP_2(x)\|\\
			\leq M_1\varepsilon^{d+1}+M_2\leq M\varepsilon^d\leq\|\cQ(x)\|.$$}{\begin{multline*}\|\cP(x)-\cQ(x)\|\leq \|\cP_1(x)\|+\|\cP_2(x)\|\\
				\leq M_1\varepsilon^{d+1}+M_2\leq M\varepsilon^d\leq\|\cQ(x)\|.
		\end{multline*}}
		Then, by the multivariate version of Rouch\'e's theorem~\cite[Theorem 2.12]{Rouche:1993}, both $\cP$ and $\cQ$ have the same number of zeros in $\bS_\varepsilon$.
	\end{proof}
	
	\ifthenelse{\mainfilecheck{1} > 0}{\section*{acknowledgments}
		Burr was supported by 
		National Science Foundation
		grant
		DMS-1913119
		and
		Simons Foundation
		collaboration grant \#
		964285.
		Leykin was supported by 
		National Science Foundation
		grant
		DMS-2001267.
	}{\begin{acks}
			Burr was supported by 
			\grantsponsor{NSF}{National Science Foundation}{http://www.nsf.gov} grant
			\grantnum{NSF}{DMS-1913119} and
			\grantsponsor{SIMONS}{Simons Foundation}{http://www.simonsfoundation.org} collaboration grant \#\grantnum{SIMONS}{964285}.
			Leykin was supported by 
			\grantsponsor{NSF}{National Science Foundation}{http://www.nsf.gov} grant
			\grantnum{NSF}{DMS-2001267}.
		\end{acks}
	}

	\bibliographystyle{plain}
	\bibliography{generalizedinflation}

\begin{thebibliography}{10}

\bibitem{BatraSharma:2019}
Prashant Batra and Vikram Sharma.
\newblock Complexity of a root clustering algorithm.
\newblock Technical Report arXiv:1912.02820, arXiv, 2019.

\bibitem{BeckerSagraloff:2017}
Ruben Becker and Michael Sagraloff.
\newblock Counting solutions of a polynomial system locally and exactly.
\newblock Technical Report arXiv:1712.05487 [cs.SC], arXiv, 2017.

\bibitem{Becker2016}
Ruben Becker, Michael Sagraloff, Vikram Sharma, Juan Xu, and Chee Yap.
\newblock Complexity analysis of root clustering for a complex polynomial.
\newblock In {\em Proceedings of the ACM on International Symposium on Symbolic
  and Algebraic Computation}, ISSAC '16, pages 71--78, 2016.

\bibitem{Becker2018}
Ruben Becker, Michael Sagraloff, Vikram Sharma, and Chee Yap.
\newblock A near-optimal subdivision algorithm for complex root isolation based
  on the {P}ellet test and {N}ewton iteration.
\newblock {\em Journal of Symbolic Computation}, 86:51--96, 2018.

\bibitem{Rouche:1993}
Carlos~A. Berenstein, Alekos Vidras, Roger Gay, and Alain Yger.
\newblock {\em Residue Currents and {B}ezout Identities}.
\newblock Progress in Mathematics. Birkh\"auser Basel, 1993.

\bibitem{inflation:2021}
Michael Burr and Anton Leykin.
\newblock Inflation of poorly conditioned zeros of systems of analytic
  functions.
\newblock {\em Arnold Mathematical Journal}, 7:431--440, 2021.

\bibitem{dedieu2001simple}
Jean-Pierre Dedieu and Mike Shub.
\newblock On simple double zeros and badly conditioned zeros of analytic
  functions of n variables.
\newblock {\em Math. Comput.}, 70(233):319--327, 2001.

\bibitem{Froberg:1985}
Ralf Fr\"{o}berg.
\newblock An inequality for {H}ilbert series of graded algebras.
\newblock {\em Mathematica Scandinavica}, 56(2):117--144, 1985.

\bibitem{GiLeSaYa:EmbedDim1}
Marc Giusti, Gr{\'e}goire Lecerf, Bruno Salvy, and Jean-Claude Yakoubsohn.
\newblock On location and approximation of clusters of zeroes: Case of
  embedding dimension one.
\newblock {\em Foundations of Computational Mathematics}, 6(3):1--57, July
  2006.

\bibitem{hao2020isolation}
Zhiwei Hao, Wenrong Jiang, Nan Li, and Lihong Zhi.
\newblock On isolation of simple multiple zeros and clusters of zeros of
  polynomial systems.
\newblock {\em Mathematics of Computation}, 89(322):879--909, 2020.

\bibitem{LVZ}
Anton Leykin, Jan Verschelde, and Ailing Zhao.
\newblock Newton's method with deflation for isolated singularities of
  polynomial systems.
\newblock {\em Theoretical Computer Science}, 359(1-3):111--122, 2006.

\bibitem{Mourrain:2020}
Angelos Mantzaflaris, Bernard Mourrain, and Agnes Szanto.
\newblock Punctual {H}ilbert scheme and certified approximate singularities.
\newblock In {\em I{SSAC}'20---{P}roceedings of the 45th {I}nternational
  {S}ymposium on {S}ymbolic and {A}lgebraic {C}omputation}, pages 336--343.
  ACM, New York, 2020.

\bibitem{OjikaExample}
Takeo Ojika.
\newblock Modified deflation algorithm for the solution of singular problems.
  i. a system of nonlinear algebraic equations.
\newblock {\em Journal of mathematical analysis and applications},
  123(1):199--221, 1987.

\bibitem{Sagraloff:2012}
Michael Sagraloff.
\newblock When {N}ewton meets {D}escartes: a simple and fast algorithm to
  isolate the real roots of a polynomial.
\newblock In {\em I{SSAC} 2012---{P}roceedings of the 37th {I}nternational
  {S}ymposium on {S}ymbolic and {A}lgebraic {C}omputation}, pages 297--304.
  ACM, New York, 2012.

\end{thebibliography}


\end{document}